\documentclass[11pt,a4paper]{amsart}

\usepackage{amsmath,amsfonts,amssymb,amsthm}
\usepackage{txfonts}
\usepackage{latexsym,bm,graphicx}
\usepackage{mathrsfs}
\usepackage{color}

\title[Boundary Regularity of Harmonic maps from $RCD(K,N)$-space to $CAT(0)$-space ]{Boundary Regularity of Harmonic maps from $RCD(K,N)$-space to $CAT(0)$-space}
\author{Hui-Chun Zhang}
\address{Department of Mathematics\\  Sun Yat-sen University\\ Guangzhou 510275\\ \newline E-mail address: zhanghc3@mail.sysu.edu.cn}
\author{Xi-Ping Zhu}
\address{Department of Mathematics\\  Sun Yat-sen University\\ Guangzhou 510275\\ \newline E-mail address: stszxp@mail.sysu.edu.cn}

 \newtheorem{theorem}{Theorem}[section]

\newtheorem{lemma}[theorem]{Lemma}
\newtheorem{corollary}[theorem]{Corollary}

\theoremstyle{definition}
\theoremstyle{remark}
 
\newtheorem{exam}[theorem]{Example}

\newtheorem{defn}[theorem]{Definition}
\newtheorem{remark}[theorem]{Remark}

\numberwithin{equation}{section}

\newcommand{\ls}{\leqslant}
\newcommand{\gs}{\geqslant}

\newcommand{\ip}[2]{\left<{#1},{#2}\right>}

\newcommand{\meas}{\mathfrak{m}}

\begin{document}

%\today

\begin{abstract} 
We establish the boundary regularity of harmonic maps from $RCD(K, N)$ metric measure spaces into $CAT(0)$ metric spaces.  
\end{abstract}

\maketitle
%\tableofcontents
%\setcounter{tocdepth}{1}

\section{Introduction}

From M. Gromov and R. Schoen \cite{GS92},  there has been growing interest in developing a theory of harmonic maps between singular metric spaces. Korevaar-Schoen \cite{KS93} and independently Jost \cite{Jost94,Jost95,Jost97}  established a general theory of Sobolev and harmonic maps with values into metric spaces with non-positive curvature in the sense of Alexandrov.  Several notions of Sobolev energy for maps and harmonic maps into/between metric targets have been introduced in \cite{KS93,Jost94,HKST01,KS03,Ohta04}. In \cite{KS93}, Korevaar and Schoen established a theory of Sobolev spaces for maps into a metric space, the source spaces are smooth manifolds. Kuwae and Shioya \cite{KS03}  extended Korevaar-Schoen's theory to the case where the source spaces are a class of singular metric spaces. Recently Gigli and Tyulenev \cite{GT21b} extended it to the case where the source spaces are $RCD(K,N)$ spaces, a class of metric measure spaces with lower bounds on the Ricci curvature in the synthetic sense. Nowadays, the theory of $RCD$-spaces develops extensively (see, for example, \cite{Stu06a,Stu06b,LV09,Gig13,AGS14a,AGS14b,EKS15,AMS16,CM21}, and so on). 
The readers can refer to the surveys  \cite{Amb18,Gig23b} (and references therein) for more about the history of
the topic.

 Let $\Omega$ be a bounded domain in an $RCD(K,N)$ space $(X,d,\meas)$ for some $K\in\mathbb R$ and $N\in [1,\infty)$, and let $(Y, d_Y)$ be a complete metric space. A map $u:\Omega\to Y$ is called an $L^2(\Omega,Y)$-map if for some $P\in Y$ the function $d_Y(P,u(x))$ is in $L^2(\Omega)$.
\begin{defn}\label{def1.1} Suppose that $\meas(X\setminus\Omega)>0$. 
 Given a map $u\in L^2(\Omega,Y),$ for each $r>0$, the \emph{approximating energy } $E^{\Omega}_{2, r}[u]$ at scale $r$  is defined as a functional on $C_0(\Omega),$ 
 $$ E^{\Omega}_{2,r}[u](\phi):=\int_\Omega\phi\cdot {\rm ks}^2_{2,r}[u,\Omega]d\meas,$$ where
  \begin{equation}\label{equ-1.1}
 {\rm ks}_{2,r}[u,\Omega](x):=
 \begin{cases}
 \left(\fint_{B_r(x)}\frac{d_Y^2(u(x),u(y))}{r^2}d\meas(y)\right)^{1/2}& {\rm if}\ B_r(x)\subset \Omega,\\
 0&{\rm otherwise}.
 \end{cases}
 \end{equation}
A map $u\in L^2(\Omega,Y) $ is called in $W^{1,2}(\Omega,Y)$ if 
$$E_2^{\Omega}(u):=\sup_{\phi\in C_0(\Omega),\ 0\ls \phi\ls 1}\Big(\limsup_{r\to 0} E^{\Omega}_{2,r}[u](\phi)\Big)<+\infty.$$
\end{defn}

  We consider the Dirichlet problem as follows. Given any $w\in W^{1,2}(\Omega,Y)$, we want to find the minimizer of 
  $$\min_{u\in W^{1,2}_{w}(\Omega, Y)} E_2^{\Omega}(u),$$
  where 
\begin{equation}
\label{equ-1.2}
W^{1,2}_{w}(\Omega,Y):=\big\{u\in W^{1,2}(\Omega,Y):\ d_Y\big(u(x),w(x)\big)\in W^{1,2}_0(\Omega)\big\}.
\end{equation}
Such an energy-minimizing map is called a \emph{harmonic map}.

When the target space is a $CAT(0)$ space, the existence and uniqueness of harmonic maps have been established in \cite{KS93,Jost97,GT21b}. In \cite{KS93}, Korevaar-Schoen obtained the interior  Lipschitz continuity for
  such harmonic maps under the assumption that the source spaces are smooth manifolds. There have been a lots of subsequent researches for the interior regularity of harmonic maps in metric settings (see, for example, \cite{Che95, Lin97,Jost97,DM08, DM10, EF01, Fug03,Fug08,Stu05,HZ17,Guo21,ZZ18,ZZZ19}).   In \cite{Lin97} and \cite{Jost97}, F. H. Lin and J. Jost proved the interior H\"older continuity  of harmonic maps from a domain of
   Alexandrov space with curvature bounded below into a $CAT(0)$ space\footnote{In \cite{Jost97}, the source spaces are locally compact metric spaces with a Dirichlet form and a uniform  Poincar\'e inequality.}.
   They conjectured that these harmonic maps should be interior Lipschitz continuous. This has been settled by the authors in \cite{ZZ18}. Recently, Mondino and Semola \cite{MS23}, and Gigli \cite{Gig23}, independently, proved the interior Lipschitz continuity of harmonic maps from
$RCD(K,N)$ spaces into $CAT(0)$ spaces.
\begin{theorem}[Mondino-Semola, Gigli]\label{thm-1.2}
 Let $\Omega\subset X$ be a bounded domain of an $RCD(K,N)$ metric measure
space $(X,d,\meas)$  for some $K\in \mathbb R$ and $N\in [1,\infty)$ with $\meas(X\setminus \Omega)>0$, and let $(Y,d_Y)$ be a $CAT(0)$ space.  Assume that $u:\Omega\to Y$ is a harmonic map. Then $u$ is  Lipschitz continuous in the interior of $\Omega$. 
\end{theorem}
However, they did not address boundary regularity questions. Recall that the boundary regularity for harmonic maps between smooth manifolds has been established by Schoen-Uhlenbeck \cite{SU83}.

 Our goal in the present paper is to study the boundary regularity for harmonic maps into/between singular metric spaces. Up to now, there are only a few answers for several special cases. 
The first is the case where the domain $\Omega$ is smooth. In 1994, Serbinowski \cite{Ser94} proved that the harmonic map $u$  is H\"older continuous near the boundary, assuming that its trace is so and that $\Omega$ satisfies a condition in terms of locally exterior cones. In particular, he proved the following
\begin{theorem}[Serbinowski]\label{thm-1.3}
 Let $\Omega$ be bounded in a smooth Riemannian manifold, and let $(Y,d_Y)$ be a $CAT(0)$ space.  Suppose that $w\in W^{1,2}(\Omega,Y)$ and  $u\in W^{1,2}_{w}(\Omega,Y)$ is a  harmonic map. Then
 
 (A)  if $\partial\Omega$ is Lipschitz (i.e., it can be locally written as a graph of a Lipschitz function) and if $w\in C^\alpha({\overline\Omega},Y)$, then $u\in C^\beta(\overline{\Omega},Y)$ for some small $\beta\in(0,\alpha]$.
 
 (B)  if  $\partial \Omega$ is of $C^1$ and  $w\in Lip({\overline\Omega},Y)$, then $u\in C^{1-\epsilon}(\overline{\Omega},Y)$ for every $\epsilon>0$. 
 \end{theorem}
\noindent Here a map $w\in C^\alpha(\overline{\Omega},Y)$ means 
 $$d_Y\big(u(x),u(y)\big)\ls Ld^\alpha d(x,y) \quad \forall x,y\in \overline\Omega$$
 for some constant $L>0$ and $\alpha\in(0,1).$
  The second is the case where the domain $\Omega$ is in an Alexandrov space with curvature bounded below. Huang-Zhang \cite{HZ17} proved that the harmonic map $u\in W^{1,2}_{w}(\Omega,Y)$  is H\"older continuous (with a small H\"older index)  near the boundary, assuming that $w\in Lip(\overline{\Omega},Y)$ and that $\Omega$ satisfies a condition in terms of locally exterior balls.

 In this paper, we will be concerned with the boundary regularity of harmonic maps on $RCD$ spaces. We first recall some classical conditions as follows.
 \begin{defn} \label{def-1.4}
 Let $\Omega $ be a domain in a metric measure space $(X,d,\meas)$. It is said to satisfy an {\emph{exterior density condition}} if  there exist two numbers $\lambda\in(0,1)$ and $R_0>0$ such that  
\begin{equation} \label{equ-1.3}
\meas(  B_r(x)\setminus \Omega)\gs \lambda \cdot \meas(B_r(x)) \quad \forall\ x\in \partial{\Omega},\ \  \forall r\in(0,R_0).
\end{equation}
It is said to satisfy a  {\emph {uniformly exterior sphere condition}} if there exists a number $R_0>0$ such that for each $x_0\in\partial \Omega$ there exists a ball
 $ B_{R_0}(y_0)$ satisfying
\begin{equation}\label{equ-1.4}
\Omega\cap B_{R_0}(y_0)=\emptyset\quad {\rm and}\quad x_0\in \partial\Omega\cap \partial B_{R_0}(y_0).
\end{equation}
\end{defn}
 From the volume doubling property, one knows if $\Omega$ satisfies a uniformly exterior sphere condition then it also satisfies an exterior density condition.
It is well-known that if $\partial \Omega$ is Lipschitz then it satisfies an exterior density condition, and  if $\partial \Omega$ is of $C^2$ on a smooth manifold then it satisfies an
exterior sphere condition.

The main result in this paper is the following boundary regularity of harmonic maps.
  \begin{theorem}\label{thm-1.5}
 Let $\Omega, Y$ be as in the above Theorem \ref{thm-1.2}. Let $w\in W^{1,2}({\Omega},Y)$ and $u\in W^{1,2}_{w}(\Omega,Y)$ be a harmonic map.  Then
 
 (A) If  $\partial\Omega$ satisfies an exterior density condition and   if $w\in C^\alpha({\overline\Omega},Y)$, then $u\in C^\beta(\overline{\Omega},Y)$ for some small $  \beta\in(0,\alpha)$.
 
 (B) If $\partial \Omega$ satisfies a uniformly exterior sphere condition and if $w\in Lip({\overline\Omega},Y)$, then $u\in C^{1-\epsilon}(\overline{\Omega},Y)$ for every $\epsilon>0$. 
  \end{theorem}

\begin{remark}\label{rem-1.6}
In the case when $Y=\mathbb R$, i.e., $u$ is a harmonic function, the conclusion of  Theorem \ref{thm-1.5}(A) has been proved by Bj\"orn in \cite{Bjo02}, under the weaker assumption that the measure $\meas$ is doubling and that $(X,d,\meas)$ supports an $L^2$-Poincar\'e inequality.
\end{remark}

%\begin{remark}\label{rem-1.7}
%Theorem \ref{thm-1.5}(B) does not cover Theorem \ref{thm-1.3}(B), because, if $\Omega$ is in a smooth manifold, the assumption about a uniformly exterior condition does not cover the $C^1$-condition.
%{remark}

\begin{remark} \label{rem-1.8}
The conclusion of Theorem \ref{thm-1.5}(B) is sharp. One can not improve the H\"older regularity of Theorem \ref{thm-1.5}(B) to Lipschitz regularity. Indeed, let us recall some known results about harmonic functions on an Euclidean space (see, for example, \cite{GT01,Ken94}). Let $D\subset \mathbb R^n$ be a bounded domain, $n\gs 2,$ and let $u\in W^{1,2}(D)$ be a  harmonic function on $D$ with boundary data $\phi \in Lip(\overline{D})$ ( i.e., $u=\phi$ on $\partial D$). Then

(1)   if $\partial D$ satisfies an exterior density condition, then $u\in C^\beta(\overline{D})$, for some small $\beta>0$,

(2) if $\partial D$ is of $C^1$, then $u\in C^{1-\epsilon}(\overline{D})$ for every $\epsilon>0$, %but not $Lip(\overline{D})$ in general,

%(3)  if  $\partial D$ is of $C^{2}$, then $f$ satisfies the logarithmic estimate \cite{HS99}
%\begin{equation}\label{equ-1.5}
%|\nabla f|(x)\ls C_{D,n}  \|\nabla \phi\|_{L^\infty(D)} \cdot \ln\left(\frac{{\rm diam}(D)}{{\rm dist}(x,\partial D)}\right)
%\end{equation}
%for almost all $x\in D$.  The following Example \ref{example-1.9} shows that (\ref{equ-1.5}) is almost sharp.
 
 (3) One can not expect  that $u\in Lip(\overline{D})$, even $\partial D$ is of $C^\infty$ (see the following Example \ref{example-1.9}).  

%(5) Under a further assumption that $g\in C^{1,\alpha}(\overline{D})$, if $D$ is of $C^{1,\alpha}$,   then $f\in Lip(\overline{D})$ (see, for example, \cite{Ldd}).
\end{remark}

\begin{exam}  \label{example-1.9} 
Consider  the upper semi-plane $\mathbb R^2_{+}=\{(x,t)|\ t\gs0, x \in\mathbb R\}$. Let $\phi$ be defined on $\mathbb R$ by
$$\phi(x)=|x|, \quad {\rm if}\quad  |x|\ls 1 \qquad {\rm and}\quad   \phi(x)=(2-|x|)^+\quad {\rm otherwise}.$$
The harmonic function on $\mathbb R^2_{+}$ with boundary data $\phi$ is uniquely determined by the Poisson integral formula,
$$u(x,s)= \int_{\mathbb R} \frac{s\phi(x)}{|x-y|^2+s^2}dy.$$
 By directly calculating, we have
 \begin{equation*}
 \begin{split}
 \frac{\partial u(0,s)}{\partial s}&=\int_{|y|\ls1}\frac{y^2-s^2}{(y^2+s^2)^2} |y|dy+\int_{2\gs |y|\gs1}\frac{y^2-s^2}{(y^2+s^2)^2}(2-|y|) dy\\
 &:= I_1(s)+I_2(s)
    \end{split}
 \end{equation*}
where 
$$I_1(s)=\int_{|y|\ls1}\frac{y^2-s^2}{(y^2+s^2)^2} |y|dy= 2\int_{0}^1\frac{y^2-s^2}{(y^2+s^2)^2} ydy  =\int_0^1\frac{y^2-s^2}{(y^2+s^2)^2} d(y^2)   =O(\log(1/s))$$
as $s\to 0^+$, and
$$|I_2(s)|\ls\int_{2\gs |y|\gs1}\frac{|y^2-s^2|}{(y^2+s^2)^2}(2-|y|) dy\ls 8 \quad \forall s\in(0,1).$$
Therefore, we conclude  
\begin{equation*}
|\nabla u(0,s)| \gs \left|\frac{\partial u(0,s)}{\partial s}\right| =O(\log(1/s)), \quad {\rm as} \ s\to0^+.
 \end{equation*}
\end{exam}

 We will organize this paper as follows.   In Section 2, we will provide some necessary materials on  $RCD(K,N)$ spaces. In Section 3,  we will review the concepts of the Sobolev energy and the harmonic maps between metric spaces.  The Section 4 is devoted
to the proof of the main Theorem \ref{thm-1.5}.  \\

\noindent{\bf Acknowledgement\ } The authors are supported by NSFC 12025109 and 12271530.

\section{Preliminaries}

 Let $(X,d,\meas)$ be a metric measure space, i.e., $(X,d)$ is a complete and separable metric space equipped with a non-negative Borel measure which is finite on any ball $B_r(x)\subset X$ and ${\rm supp}(\meas)=X$. We will denote by $B_r(x_0)$ the open ball in $X$. Given an open domain $\Omega\subset X$, we denote by $Lip(\Omega)$ (resp.   $C_0(\Omega)$, $Lip_0(\Omega)$, $Lip_{\rm loc}(\Omega)$) the space of Lipschitz continuous (resp. continuous with compact support, Lipschitz continuous with compact support, locally Lipschitz continuous) functions on $\Omega$, and by $L^p(\Omega):=L^p(\Omega, \meas)$ for any $p\in[1,+\infty]$. For any $f:E\to \mathbb R$ on a measurable subset $E\subset X$, we denote by 
 $$\fint_Efd\meas:=\frac{1}{\meas(E)}\int_Efd\meas.$$

 Let $f\in Lip_{\rm loc}(X)$, the \emph{pointwise Lipschitz constant} or \emph{slope} is given by
 \begin{equation*}
 {\rm Lip}f(x):=
 \begin{cases}
 \limsup_{y\to x}\frac{|f(y)-f(x)|}{d(x,y)}= \lim_{r\to 0^+}\sup_{y\in B_r(x)}\frac{|f(y)-f(x)|}{r} &   {\rm if}  x\ {\rm is\ not \ isolated},\\
 0&\quad {\rm otherwise}.
 \end{cases}
 \end{equation*}
 The \emph{Cheeger energy} ${\rm Ch}: L^2(X)\to [0,+\infty]$ is  
 $${\rm Ch}(f):=\frac{1}{2}\inf\Big\{\liminf_{j\to +\infty}\int_X({\rm Lip}f_j)^2d\meas:\ f_j\in Lip_{\rm loc}(X)\cap L^2(X),\ \ f_j\overset{L^2}{\to }f\Big\}.$$
For any $f\in W^{1,2}(X)$, it was showed \cite{Che99,AGS14a} that there exists a  {\rm minimal weak upper gradient} $|\nabla f|$ such that 
 $${\rm Ch}(f)=\frac{1}{2}\int_X|\nabla f|^2d\meas,$$
 and for any $f\in Lip(X) $ that $|\nabla f|={\rm Lip}f$  $\meas-$a.e. in $X$.  
 The Sobolev space $W^{1,2}(X):= \{f\in L^2(X): \ {\rm Ch}(f)<+\infty\}$.

There are several equivalent definitions for the \emph{Riemannian curvature-dimension condition} on metric measure spaces \cite{EKS15,AMS16,CM21,AGS14b}.  We state one as follows.
\begin{defn}\label{def-2.1}
Let $K\in\mathbb R$ and $N\in(1,+\infty]$, a metric measure space $(X, d, \meas)$ is called an $ RCD(K, N)$-space  
if the following  conditions hold:
   \begin{enumerate}
       \item  The Sobolev space $W^{1,2}(X)$ is a Hilbert space.  In this case, it was proved \cite{Gig15} that  for any $f,g\in W^{1,2}(X)$ the limit
  \begin{equation}\label{equ-2.1}
  \ip{\nabla f}{\nabla g}:=\lim_{\epsilon\to0^+}\frac{|\nabla (f+\epsilon g)|^2-|\nabla f|^2}{2\epsilon}
  \end{equation}
    exists and is in $L^1(X)$. This provides a canonical Dirichlet form 
    $\mathscr E(f,g):=\int_X\ip{\nabla f}{\nabla g}d\meas$ for any $ f,g\in W^{1,2}(X).$
     We denote by $\Delta$   the infinitesimal generator  of $\mathscr E$ and by   $\{H_t:=e^{t\Delta}\}_{t\gs0}$ the corresponding  semi-group (heat flow). The domain of $\Delta$ is denoted by $D(\Delta)$. 
             \item  For any $f\in D(\Delta )$ with $\Delta f\in W^{1,2}(X)$, it holds  the weak Bochner inequality
 \begin{equation}\label{equ-2.2}
 \int_X|\nabla f|^2\Delta gd\meas\gs \int_X\Big( \frac{(\Delta f)^2}{N} +\ip{\nabla f}{\nabla \Delta f}+K|\nabla f|^2\Big)gd\meas
 \end{equation}
     for any $g\in D(\Delta )\cap L^\infty(X)$ with $\Delta g\in L^\infty(X)$ and $g\gs 0$.
 \item   There  are some point $x_0\in X$ and  constants $c_1,c_2>0$   such that $\meas(B_r(x_0))\ls c_1 e^{c_2r^2}$ for all $r>0$.
 \item For any $f\in W^{1,2}(X)$ with $|\nabla f|\ls 1$ $\meas-$a.e. in $X$, there exists a representative $g\in Lip(X)$ with Lipschitz constant 1 such that $g=f$ $\meas-$a.e. in $X$.
  \end{enumerate}
\end{defn}

 Let $(X,d,\meas)$ be an $RCD(K,N)$ space for some $K\in\mathbb R$ and $N\in [1,+\infty)$. It is  a geodesic space (see \cite{AGS14b}), i.e., for  any two points $x,y\in X$, there exists a curve $\gamma$ connecting them such that the length $L(\gamma)=d(x,y).$ Recall that the length of a continuous curve  $\gamma:[a,b]\to X$ is defined by
\begin{equation*}
L(\gamma)=\sup_{a=a_0<a_1<\cdots<a_m=b}\sum_{i=0}^{m-1}d\big(\gamma(a_i),\gamma(a_{i+1})\big).
\end{equation*}
 Any closed ball $\overline{B_R(x)}\subset X$ satisfies a standard assumption that  $(\overline{B_R(x)}, d, \meas)$ has a volume doubling property and supports an $L^2$-Poincar\'e inequality, both the doubling constant $C_D$ and the Poincar\'e constant $C_P$ depend only on $N,K,R$  (see  \cite{Stu06b,LV09}).
 
  The semi group $\{H_t\}_{t\gs0}$ on $L^2(X)$  admits a heat kernel $p_t(x,y)$ on $X$  (see \cite{Stu95,AGS14a}) such that 
$$H_tf(x)=\int_Xp_t(x,y) f(y)d\meas(y),\quad \forall f\in L^2(X).$$
The heat kernel  is Lipschitz continuous (see \cite{HZ20}) and  satisfies the stochastic completeness, i.e.,
\begin{equation}\label{equ-2.3}
\int_Xp_t(x,y)d\meas(y)=1 
\end{equation} 
for all $x\in X$.   An upper bound of $p_t(x,y)$ was proved in \cite{JLZ16} that 
\begin{equation*} 
p_t(x,y)\ls \frac{C_{N,K}}{m\big(B_{\sqrt t}(x)\big) }\exp\Big(-\frac{d^2(x,y)}{5t}+C_{N,K}\cdot t\Big)
\end{equation*}
for some constant $C_{N,K}>0$ depending only on $N$ and $K$.  Erbar–Kuwada–Sturm \cite{EKS15} proved the Bakry-Ledoux' estimate: for  every $f\in W^{1,2}(X)$ and every $t>0$ it holds  
 \begin{equation}\label{equ-2.4}
 |\nabla (H_tf)|^2+\frac{4Kt^2}{N(e^{2Kt}-1)}|\Delta H_tf|^2\ls e^{-2Kt}H_t(|\nabla f|^2)\quad \meas{\rm-a.e.\ \ in}\ X. 
 \end{equation}

\begin{defn}[Local Sobolev Space] 
 Let $\Omega\subset X$ be an open set. A  function $f\in L^2_{\rm loc}( \Omega)$ belongs to
$W^{1,2}_{\rm loc}(\Omega)$, provided, for any    $\chi\in Lip_0(\Omega)$   it holds
$f\chi\in W^{1,2}(X)$, where $f\chi$ is understood to be  $0$ outside from $\Omega$. In this case, the function $|\nabla f| : \Omega\to[0,\infty]$ is $\meas$-a.e. defined by
$$|\nabla f| := |\nabla (\chi f)|,\ \ \meas-a.e. \,\mathrm{on} \ \{\chi=1\},$$
 for any $\chi$ as above. 
\end{defn}
The space 
$$W^{1,2}(\Omega):=\{f\in W^{1,2}_{\rm loc}(\Omega):\  f, |\nabla f|\in L^2(\Omega)\}$$
 with the norm  
 $$\|f\|_{W^{1,2}(\Omega)}:=(\|f\|^2_{L^2(\Omega)}+\||\nabla f|\|^2_{L^2(\Omega)})^{1/2}.$$  Finally, the space $W_0^{1,2}(\Omega)$ is defined as the $W^{1,2}(\Omega)$-closure of the space of functions $f\in W^{1,2}(\Omega)\cap C(\Omega)$ with ${\rm supp}(f)\subset \Omega.$

Recall the notion of the measure-valued  Laplacian in  \cite{Gig13}.
\begin{defn}\label{def-2.3}
  Let $\Omega\subset X$ be an open subset. A function $g\in W^{1,2}_{\rm loc}(\Omega)$ is called to be in   $D({\bf\Delta},\Omega)$ if there exists a signed Radon measure $\mu$ on $\Omega$ such that for any
 $\phi\in Lip_0(\Omega)$ it holds
\begin{equation}\label{equ-2.5}
-\int_\Omega\ip{ \nabla f}{\nabla\phi} d\meas =\int_\Omega  \phi d\mu.
\end{equation}
 If such $\mu$ exists,  then it must be unique. In this case, we will write ${\bf\Delta} g=\mu$. 
 
 \end{defn}
\noindent It was shown \cite{Gig13} that the operator ${\bf \Delta}$ is linear, and both the Leibniz rule and the chain rule hold.
In general, the measure-valued Laplacian ${\bf\Delta} f$ for a function $f\in W^{1,2}_{\rm loc}(\Omega)$ may not be absolutely continuous with respect to $\meas$. We consider its Radon-Nikodym decomposition
$${\bf \Delta}f=({\bf \Delta }f)^{\rm ac}\cdot \meas+({\bf \Delta}f)^{\rm sing}.$$
 When $\Omega=X$, the  compatibility of  ${\bf \Delta} $ and the $\Delta$  has been shown in \cite{Gig15}, which states that  {if} $f\in W^{1,2}(X)$ and $g\in L^2(X)$, then
 $$ {\bf \Delta} f=g\cdot \meas \Longleftrightarrow f\in  D(\Delta ) \ {\rm and}\ \Delta f=g.$$
 
Let $f\in D({\bf \Delta},\Omega)$ and $g\in L^1_{\rm loc}(\Omega)$. We denote by  ${\bf \Delta}f\gs g$ as measures (or in the sense of distributions)  if  $-\int_\Omega\ip{ \nabla f}{\nabla\phi} d\meas\gs \int_\Omega g\phi d\meas$ for all nonnegative $\phi\in Lip_0(\Omega)$. 
When $f\in D({\bf \Delta},\Omega)$ and  ${\bf \Delta}f\gs 0$ as measures, we have $f^+:=\max\{f,0\} \in  D({\bf \Delta},\Omega)$ and  ${\bf \Delta}f^+\gs 0$ as measures. We have also the Maximum principle as follows (see \cite{Che99}). Let $f\in D({\bf \Delta},\Omega)$ and  ${\bf \Delta}f\gs 0$ as measures. If $f\ls c$ on $\partial \Omega$ in the sense of $(f-c)^+\in W^{1,2}_0(\Omega)$  then $f\ls c$ almost $\meas$-a.e. in $\Omega.$ In particular, if $f\in W^{1,2}(\Omega)\cap C(\overline{\Omega})$,   if $f\ls c$ on $\partial \Omega$, and if ${\bf \Delta} f\gs0$ as measures, then $f\ls c$ on $\Omega.$

 Given a function $h\in L^2(\Omega)$ and $g\in W^{1,2}(\Omega)$, we can solve the (relaxed) Dirichlet problem of the Poisson equation  (see, for example,  \cite{Che99})
   \begin{equation*}
 \begin{cases}{\bf \Delta}f=h \\ 
f-g\in W_0^{1,2}( \Omega).
 \end{cases}
 \end{equation*}
  The following boundary regularity result for the harmonic functions has been given in \cite{Bjo02}.
\begin{lemma} 
\label{lem-2.4}
Let $\Omega$ be a bounded domain of $(X,d,\meas)$. Assume that
 $ \Omega $ satisfies an exterior density condition. Then for each harmonic function $f\in W^{1,2}(\Omega)$ such that $f-g\in W^{1,2}_0(\Omega)$ for some $g\in C^\alpha(\overline\Omega)\cap W^{1,2}(\Omega)$. Then,  after a redefinition on a
set of measure zero, $f$ is  $C^\beta$-continuous at each $x_0\in\partial\Omega$, for some small $\beta\in (0,\alpha)$. Moreover,  there exist    constants  $r_0\in (0,R_0)$ and $C>0$  such that 
 \begin{equation*}
\sup_{x\in B_r(x_0)} |f(x)-f(x_0)|\ls Cr^\beta
 \end{equation*}
 for all $r\in(0,r_0)$ and all $x_0\in\partial \Omega,$ where $R_0$ is the number given in the definition of exterior density condition, and $C$ depends only on $N,K, {\rm diam}(\Omega),$ and a modulus of the H\"older continuity of $g$, i.e., a constant $C_1>0$ such that 
 $|g(x)-g(x_0)|\ls C_1 d^\alpha(x_0,x)$ for all $x_0,x\in \overline{\Omega}$. (See Equation (16) in \cite{Bjo02} for the details). 
 
 \end{lemma}

The following corollary of the Laplacian comparison theorem \cite{Gig13} will be used to construct the barriers.  
 
\begin{lemma}\label{lem-2.5}
For each $y_0\in X$, the distance function $d_{y_0}(x):=d(x,y_0)$ is in $ D({\bf \Delta}, X\setminus \{y_0\})$, and  there exists a number $m_{N,K}\gs1$ such that  
\begin{equation}\label{equ-2.6}
  {\bf \Delta }d^{-m}_{y_0}\gs  \frac{m^2}{2 R^{m+2}}\quad {\rm on }\quad B_{R}(y_0)\setminus \{y_0\}
  \end{equation}
for all $R\in(0,1)$ and all $m\gs m_{N,K}$, in the sense of distributions.
\end{lemma}
\begin{proof}
From the Laplacian comparison, we have  $d_{y_0}(x) \in D({\bf \Delta}, X\setminus \{y_0\})$ and 
$$d_{y_0}\cdot {\bf \Delta }d_{y_0}\ls C_{N,K}\quad {\rm on }\quad B_1(y_0)\setminus \{y_0\}$$
in the sense of distributions, for some constant $C_{N,K}$ depending only on $N,K$.

 Notice that $d^{-m}_{y_0}$ is in $Lip_{\rm loc}(B_1(y_0)\setminus\{y_0\})$.  It follows from the Chain rule \cite{Gig13} that
\begin{equation*}
\begin{split}
{\bf \Delta }d^{-m}_{y_0}&=(-m)d_{y_0}^{-m-2}\cdot\big(d_{y_0}\cdot  {\bf \Delta }d_{y_0} -(m+1)|\nabla d_{y_0}|^2\big)\\
&\gs - m\big( C_{N,K} -(m+1) \big) d_{y_0}^{-m-2}\\
&=(m^2-mC_{N,K}+m)\frac{1}{d_{y_0}^{m+2}}\gs  \frac{m^2-mC_{N,K}+m }{R^{m+2}}\end{split}
\end{equation*}
 on $ B_{R}(y_0)\setminus \{y_0\}$ in the sense of distributions,  where we have used $d_{y_0}\ls R$ and $|\nabla d_{y_0}|=1$  almost everywhere in $B_{R}(y_0).$
  Now  the desired estimate holds provided  $m \gs m_{N,K}:=\max\{2C_{N,K},1\}$.
 \end{proof}

\section{Sobolev spaces for maps with metric targets and harmonic maps}

Let $\Omega$ be a bounded open domain of  an $RCD(K,N)$ space $(X,d,\meas)$ for some $K\in\mathbb R$ and $N\in [1,+\infty)$ such that $\meas(X\backslash \Omega)>0$, and let $(Y,d_Y)$ be a complete metric space.

 A Borel measurable map $u:\ \Omega\to Y$ is said to be in the space $L^2(\Omega,Y)$ if it has a separable range and, for some (hence, for all) $P\in Y$,
$$\int_\Omega d^2_Y\big(u(x),P\big)d\meas(x)<\infty.$$
We equip with a distance in $L^2(\Omega,Y)$ by 
$$d^2_{L^2}(u,v):=\int_\Omega d^2_Y\big(u(x),v(x)\big)d\meas(x) $$
 for all $  u,v\in L^2(\Omega,Y).$ 

Several notions of Sobolev energy for maps into metric targets have been introduced in \cite{KS93,Jost95,HKST01,KS03,Ohta04,GT21b}. We will recall the notions in \cite{KS93,KS03,GT21b}.
\begin{defn}
 Given a map $u\in L^2(\Omega,Y),$ for each $r>0$, the \emph{approximating energy } $E^{\Omega}_{2, r}[u]$ at scale $r$  is defined as a functional on $C_0(\Omega),$ 
 $$ E^{\Omega}_{2,r}[u](\phi):=\int_\Omega\phi\cdot {\rm ks}^2_{2,r}[u,\Omega]d\meas,$$ where ${\rm ks}_{2,r}[u,\Omega](x)$ is given in (\ref{equ-1.1}).
A map $u\in L^2(\Omega,Y) $ is called in $W^{1,2}(\Omega,Y)$ if 
$$E_2^{\Omega}(u):=\sup_{\phi\in C_0(\Omega),\ 0\ls \phi\ls 1}\Big(\limsup_{r\to 0} E^{\Omega}_{2,r}[u](\phi)\Big)<+\infty.$$
\end{defn}
 The following properties  have been proved in   \cite{GT21b}:
 \begin{enumerate}
 \item  if $u\in W^{1,2}(\Omega,Y)$,  there is a function $e_u:=e_2[u,\Omega]\in L^1(\Omega)$, called \emph{energy density}, such that 
    \begin{equation}\label{equ-3.1}
       E^\Omega_2(u):= \int_\Omega e_ud\meas,  
              \end{equation}         
       \item \indent If $Y=\mathbb R$, then the above  space $W^{1,2}(\Omega,\mathbb R)$
is equivalent to the   Sobolev space $W^{1,2}(\Omega)$ given in \S 2.    
 \item if  $u\in W^{1,2}(\Omega,Y)$ and $\phi\in Lip(Y,Z)$, then $\phi\circ u\in W^{1,2}(\Omega, Z)$. In particular, $d_Y(P,u(x)) \in W^{1,2}(\Omega) $ and $|\nabla d_Y(P,u(x))| \ls e^{1/2}_u$ almost everywhere.
\end{enumerate}

 Given any $w\in W^{1,2}(\Omega,Y)$,  we consider the minimizer of 
  $$\min_{u\in W^{1,2}_{w}(\Omega, Y)} E_2^{\Omega}(u),$$
  where $W^{1,2}_{w}(\Omega, Y)$ is given in (\ref{equ-1.2}).
  Such an energy-minimizing map is called a \emph{harmonic map}.

Let us recall the concept of non-positive curvature in the sense of Alexandrov. 
   
\begin{defn}[see, for example, \cite{BBI01,EF01}]\label{def-3.2}
A complete space $(Y,d_Y)$ is called to be a   $CAT(0)$ space if it is a geodesic space and if it holds the following comparison property. 
  Given any triangle $\triangle PQR\subset Y$  and point $S\in QR$ with
 $$d_Y(Q,S)=d_Y(R,S)= \frac{1}{2} d_Y(Q,R),$$
then for a  triangle $\triangle \bar P\bar Q\bar R$ in 2-dimensional plane $\mathbb R^2$  such that
$$ d_Y(P,Q)=|\bar P-\bar Q|,\quad d_Y(Q,R)=|\bar Q-\bar R|,\quad  d_Y(R,P)=|\bar R-\bar P|$$
and  a point $\bar S\in \bar Q\bar R$ with 
 $$|\bar Q-\bar S|=|\bar R-\bar S|=\frac{1}{2}|\bar Q-\bar R|,$$
we have 
$$d_Y(P,S)\ls |\bar P-\bar S|.$$
\end{defn}

It has been proved in \cite[Theorem 6.4]{GT21b} that when $Y$ is a $CAT(0)$-space, there always exists a (unique) harmonic map $u\in W^{1,2}_w(\Omega,Y)$ for any $w\in W^{1,2}(\Omega,Y)$.  Recently, Mondino-Semola  \cite{MS23} and Gigli \cite{Gig23}, independently, obtained the following interior regularity result. 
\begin{theorem}[\cite{MS23,Gig23}]\label{thm-3.3}
Let $\Omega $ be a bounded domain in an $RCD(K,N)$ space for some $K\in \mathbb R$ and $N\in [1,\infty)$ with $\meas(X\setminus\Omega)>0$, and let 
  Y be a $CAT(0)$ space. Suppose that $u:\Omega\to Y$ is a harmonic map. Then the following two properties are satisfied:
  
  (1) $u$ is locally Lipschitz continuous in $\Omega$, and moreover 
 \begin{equation}\label{cheng-yau}
 \sup_{x,y \in B_r(z_0),\ x\not=y}\frac{d_Y\big(u(x),u(y)\big)}{d(x,y)}\ls \frac{C_{N,K,{\rm diam}(\Omega)}}{r}\inf_{O\in Y}\sqrt {\fint_{B_{2r}(z_0)}d^2_Y\big(u(\cdot),O\big)d\meas}
   \end{equation}
for any ball $B_r(z_0)$ such that $B_{2r}(z_0)\subset\Omega$;

  (2) for any $P\in Y$  it holds
  \begin{equation}\label{equ-3.2}
  {\bf \Delta} d_Y\big(u(x),P\big)\gs 0 
  \end{equation}
  in the sense of distributions.  
  \end{theorem}
\begin{proof}
The conclusion (1) is proved in \cite{MS23,Gig23}, and the explicit estimate (\ref{cheng-yau}) is in \cite{Gig23}.

To show (2), we fix any $\epsilon>0$ and put
$$f_\epsilon(x):=(d^2_Y\big(u(x),P\big)+\epsilon^2)^{1/2}\gs \epsilon.$$
Then we have $f^2_\epsilon=d^2_Y\big(u(x),P\big)+\epsilon^2$,
\begin{equation}\label{equ-3.4}
 |\nabla f_\epsilon|=\frac{d_Y\big(u(x),P\big)}{f_\epsilon}|\nabla d_Y\big(u(x),P\big)|\ls |\nabla d_Y\big(u(x),P\big)|
 \end{equation}
almost everywhere in $\Omega$, and, by the chain rule \cite{Gig13}, 
$$2f_\epsilon\cdot {\bf{\Delta}} f_\epsilon+2|\nabla f_\epsilon|^2={\bf \Delta} f^2_\epsilon={\bf\Delta} d^2_Y\big(u(x),P\big) $$
in the sense of distributions. 
Thus, by combining with 
  \begin{equation*} 
{\bf \Delta} d^2_Y\big(u(x),P\big)\gs 2e_u
  \end{equation*}
as measures (see \cite{Jost97,MS23,Gig23}),  and   $|\nabla d_Y(P,u(x))|^2\ls e_u$, we have
    \begin{equation*} 
f_\epsilon\cdot {\bf \Delta}f_\epsilon \gs0
  \end{equation*}
as measures. Notice that   $d_Y\big(u(x),P\big)   \in C(\Omega)$  (see \cite{Jost97,MS23,Gig23}). Hence, we  have   $f_\epsilon\in C(\Omega)$ and $f_\epsilon\gs\epsilon$. Therefore, we conclude ${\bf \Delta}f_\epsilon \gs0$ as measures for any $\epsilon>0$. That is,
\begin{equation}\label{equ-3.5}
-\int_\Omega\ip{\nabla  f_\epsilon}{\nabla \phi }d\meas\gs 0,\quad \forall \phi\in Lip_0(\Omega),\ \ \phi\gs0.
\end{equation}
From the fact $f_\epsilon^2=d^2_Y\big(u(x),P\big)+\epsilon^2$ and (\ref{equ-3.4}), we know that $\{f_\epsilon\}_{0<\epsilon<1}$ is uniformly bounded  in $W^{1,2}_{\rm loc}(\Omega)$. (i.e. for each $\Omega' \subset\subset \Omega$, $\{f_\epsilon\}_{0<\epsilon<1}$ is bounded in $W^{1,2}(\Omega')$.) Then there exists a sequence of numbers  $\{\epsilon_j\}_{j\in\mathbb N}$, $\epsilon_j\to0^+$, such that $f_{\epsilon_j}$ converges weakly to $d_Y\big(u(x),P\big)$ in $W_{\rm loc}^{1,2}(\Omega)$. Letting now $\epsilon_j\to 0^+$ in (\ref{equ-3.5}), we conclude that ${\bf \Delta}d_Y\big(u(x),P\big)\gs 0$ in the sense of distributions. 
\end{proof}

 \section{Boundary regularity  of harmonic maps}
 We will give a proof of Theorem \ref{thm-1.5} in this section. Let $\Omega $ be a bounded domain of an $RCD(K,N)$ metric measure
space $(X,d,\meas)$  for some $K\in \mathbb R$ and $N\in [1,\infty)$ with $\meas(X\setminus \Omega)>0$.
 Without loss of the generality, we assume always $K\ls0$. 
 
 \subsection{Boundary estimates for harmonic functions}
In this subsection, we will deal with a special case, $Y=\mathbb R$, i.e., the harmonic function.
We need a simple lemma as follows. 
 
 \begin{lemma}\label{lem-4.1}
 Let $f\in L^2(X)$ be a Lipschitz function with $|\nabla f|\ls L$ for some $L>0$. There exists a constant $C_1=C_1(N,K)\gs 0$, depending only on $N,K$, such that    
  \begin{equation}\label{equ-4.1}
 |\nabla  H_tf|(z)\ls C_1  L, 
 \end{equation} \begin{equation}\label{equ-4.2}
 |\Delta H_tf|(z)\ls C_1  Lt^{-1/2},
 \end{equation}
  \begin{equation}\label{equ-4.3}
 |  H_tf(z)-f(z)|\ls C_1  Lt^{1/2}
 \end{equation} for almost all $z\in X$ and all $0<t\ls 1.$ 
  \end{lemma}
 
 \begin{proof} Since $f\in L^2(X)$, we know $H_tf\in D(\Delta)$. 
 From the Bakry-Ledoux's estimate (\ref{equ-2.4}) and the $L^\infty(X)$-contraction of $H_t$, we have
 $$|\nabla H_tf|^2\ls e^{-2K}H_t(|\nabla f|^2)\ls e^{-2K}L^2$$
and 
 $$|\Delta H_tf|^2\ls \frac{N(e^{2Kt}-1)}{4Kt^2} H_t(|\nabla f|^2)\ls \frac{N(e^{2Kt}-1)}{4Kt^2} L^2$$
 for almost all  $z\in X$ and all $0<t\ls1$. It follows  the desired estimate (\ref{equ-4.1}) and (\ref{equ-4.2}), by using  an elementary inequality  
 $$ e^s-1 \ls  s,\quad \forall s\in \mathbb R$$
and letting $s= 2Kt.$ (keeping in mind $K\ls0$). 
Finally, we have   (\ref{equ-4.3})  by
$$|H_tf(z)-f(z)|\ls \int_0^t\left|\frac{\partial H_\tau f(z)}{\partial s}\right|d\tau=\int_0^t\left|\Delta H_\tau f(z) \right|d\tau\ls  CLt^{1/2}$$
for almost all $z\in X$ and $0<t\ls1$.
The proof is finished.
\end{proof}
 
We consider the growth of a harmonic function near the boundary.
\begin{lemma}\label{lem-4.2}
Let $\Omega\subset X$ be a bounded domain and satisfy a uniformly exterior condition with constant $R_0$. Suppose that $f\in C(\overline{\Omega})\cap W^{1,2}(\Omega)$ is a harmonic function on $\Omega$ with boundary data $g\in Lip(\overline{\Omega})$. Suppose that $g$ has a Lipschitz constant $L$ and that $g(z_0)=0$ for some $z_0\in \overline\Omega.$   Given any $\epsilon\in (0,1)$, there exists $R_\epsilon \in(0,\frac{1}{2}\min\{R_0,1\})$ (depending only on  $\epsilon, N,K$ and ${\rm diam}(\Omega)$) such that  
 \begin{equation}\label{equ-4.4}
M_{x_0}(r)\ls \max\big\{C_2 M_{x_0}(2r^{\frac{1}{1+\epsilon}})\cdot r^{\frac{\epsilon}{1+\epsilon}} +C_2Lr, \ Lr^{1-\epsilon}\big\}
\end{equation}
for all  $r\in(0,R_\epsilon)$ and for all $x_0\in\partial\Omega$,
where 
$$M_{x_0}(r):=\max\big\{\sup_{x\in B_r(x_0)\cap \Omega} \left(f(x)-f(x_0)\right),\  L r^{1-\epsilon}\big\},$$
and the constant $C_2\gs 1$ depends only on $N,K$ and ${\rm diam}(\Omega)$.
 \end{lemma}
 
 \begin{proof}
 Let $\tilde g$ be a Lipschitz extension of $g$ on $X$ with $\tilde g\in L^2(X)$ and the same Lipschitz constant $L$. Let $H_t\tilde g$ be the heat flow with $H_0\tilde g=\tilde g$. 
 For any $t\in(0,1)$,  we solve the Dirichlet  problems
 $$
 {\bf \Delta} f_t =0\quad {\rm on }\quad \Omega\quad {\rm with}\quad f_t-H_t\tilde g\in W^{1,2}_0(\Omega).$$
Since $\Omega$ satisfies a uniformly exterior sphere condition, in particular, an exterior density condition, we know from Lemma \ref{lem-2.4}  that $f, f_t\in C(\overline{\Omega}).$
 Notice that $f_t-f$ is harmonic on $\Omega$ with the boundary value $H_t\tilde g-g$. It follows
$$|f_t-f|\big|_{\partial \Omega}=|H_t\tilde g-\tilde g|\big|_{\partial\Omega}\ls C_1Lt^{1/2}$$
for all $t\in (0,1)$, where $C_1$ is given in Lemma \ref{lem-4.1}.  The maximum principle yields that 
\begin{equation}
\label{equation+1}
\left| f_t-f\right|\ls C_1Lt^{1/2}\quad \forall x\in\Omega,\ \ \forall t\in(0,1).
\end{equation} 
Without loss of the generality, we can assume that $C_1\gs\max\{1, {\rm diam}(\Omega)\}$. 
Here and in the sequel of this proof, we denote $C_1, C_2,\cdots, $ as constants depending only on $N, K$ and ${\rm diam}(\Omega)$. 

 Putting 
\begin{equation}\label{equ-4.5}
v_t=f_t-H_t\tilde g\quad {\rm on}\ \Omega,
\end{equation}
it follows that  $v_t\in C(\overline{\Omega})$ for all $t\in (0,1)$, and for any $x\in B_{2R}(x_0)\cap {\overline\Omega}$, 
$$v_t(x)\ls f(x)-g(x)+2C_1Lt^{1/2}\ls M_{x_0}(2R)+2LR+2C_1Lt^{1/2},$$
where we have used the definition of $M_{x_0}(2R)$, (\ref{equation+1}), (\ref{equ-4.3}), $g(x_0)=f(x_0)$,  and that $g$ is $L$-Lipschitz on $\Omega$. From the definition of $M_{x_0}(2R)$, we have $M_{x_0}(2R)\gs L(2R)^{1-\epsilon}\gs 2L R$ provided $2R\ls 1$.
 Therefore,    for any $x\in B_{2R}(x_0)\cap {\overline\Omega}$ and $t<1$ we have
\begin{equation}\label{equ-4.6}
v_t(x)\ls 2 M_{x_0}(2R)+2C_1Lt^{1/2} 
\end{equation}
provided $2R\ls 1$.
 
  Fix some $m\gs m_{N,K}$ given in Lemma \ref{lem-2.5} such that $\frac{2^m\cdot C_1^2}{m^2}\gs 8$, (say, for example, $m=\max\{m_{N,K},12\}$ and notice that $C_1\gs1$). Let $\tilde R_\epsilon$ be determined by
\begin{equation}\label{equ-4.7}
\frac{4}{C^2_1 \tilde R^{2\epsilon}_\epsilon}=\frac{2^m}{m^2}.\quad \qquad ({\rm It\ implies}\ \ 2\tilde R_{\epsilon}\ls 1.)
\end{equation}

For any $R<\tilde R_\epsilon$, we choose $t$ such that
 \begin{equation}\label{equ-4.8}
\frac{4M}{R^2}\cdot \frac{m^2}{2^m}=C_1Lt^{-1/2},\qquad {\rm and \ let}\quad B=\frac{4M R^m}{2^{m-1}},
 \end{equation}
 where $M:=M_{x_0}(2R).$ 
  Noticing that $\frac{R}{M}\ls  \frac{R}{L(2R)^{1-\epsilon}} \ls \frac{R^\epsilon}{L}$, we have
\begin{equation}\label{equ-4.9}
t^{1/2}=\frac{C_1LM}{4}\cdot\frac{R^2}{M^2}\cdot\frac{2^m}{m^2}\ls \frac{M}{C_1L}\left(\frac{R^{2\epsilon}}{\tilde R_\epsilon^{2\epsilon}}\right)  \ls \frac{M}{C_1L}.
\end{equation}
We first check that   $t\ls 1$. Indeed,  it follows from  $C_1\gs \max\{1,{\rm diam}(\Omega)\}$ and $M=M(2R)\ls\max\{{\rm osc}_{\overline\Omega} f,L\}\ls \max\{{\rm osc}_{\overline\Omega} g,L\}\ls L\cdot\max\{ {\rm diam}(\Omega),1\}$ (since the maximal principle and $g(z_0)=0$ for some $z_0\in \overline\Omega$).

Since $\Omega$ satisfies a uniformly exterior sphere condition (and $X$ is a geodesic space), for any $R\in(0,R_0/2)$, there exists a point $y_0$ such that 
$$d(y_0,\partial\Omega)=d(y_0,x_0)=R/2.$$
We consider a  barrier 
$$h(x)=B\big( (R/2)^{-m}-d_{y_0}^{-m}(x)\big)$$
on $B_R(y_0).$ From Lemma \ref{lem-2.5} and (\ref{equ-4.8}), we have 
$${\bf \Delta}h\ls -B\frac{m^2}{2R^{m+2}}=-\frac{4MR^m}{2^{m-1}}\cdot \frac{m^2}{2R^{m+2}}=-\frac{4M}{R^2}\cdot\frac{m^2}{2^m}=-C_1Lt^{-1/2}$$
on $B_R(y_0).$ Consider   $\Omega':=\Omega\cap B_{R}(y_0)$, then $\partial \Omega'=\Gamma_1\cup\Gamma_2$, where $\Gamma_1=\partial B_{R}(y_0)\cap \overline{\Omega}$ and $\Gamma_2=\partial \Omega\cap \overline{B_{R}(y_0)}.$ It follows from (\ref{equ-4.6}) and (\ref{equ-4.9}) that $$ v_t\ls 4M\ls \frac{2^m-1}{2^{m-1}}\frac{4MR^m}{R^m}=h\quad {\rm on}\ \Gamma_1,$$
   and
 $$v_t=0\ls h\quad {\rm on}\ \Gamma_2.$$
By (\ref{equ-4.2}), we have
  $${\bf \Delta} v_t=-\Delta H_t\tilde g\gs -C_1Lt^{-1/2},\quad v_t|_{\partial \Omega}=0$$ for all $t\in(0,1)$.
Using the maximum principle, we conclude that $v_t\ls h$ on $\Omega'.$ Hence, noticing that $h(x_0)=0$, we have 
\begin{equation}\label{equ-4.10}
v_t(x)\ls h(x)\ls d(x_0,x)\cdot |\nabla h|_{L^\infty(\Omega')}  \ls d(x_0,x)\cdot \frac{mB}{(R/2)^{m+1}}= 16Mm \frac{d(x,x_0)}{R}
\end{equation}
for all $x\in B_{R/2}(x_0)\cap \Omega$. Now for any $r<R/2$ and $x\in B_r(x_0)$, we conclude that
\begin{equation*}
\begin{split}
f(x)&\ls f_t(x)+C_1Lt^{1/2}=v_t(x)+H_t\tilde g(x)+ C_1Lt^{1/2}\\
&\ls 16m M\frac{r}{R}+g(x)+2C_1Lt^{1/2}\\
&\ls 16m M\frac{r}{R}+Lr+g(x_0)+2\frac{(C_1L)^2R^2}{4M}\cdot \frac{2^m}{m^2},
\end{split}
\end{equation*}
where we have used (\ref{equation+1}), (\ref{equ-4.3}), (\ref{equ-4.10}), (\ref{equ-4.5}) and the Lipschitz continuous of $g$. Therefore, by using $f(x_0)=g(x_0)$ and $R/M\ls R^\epsilon/L$, we conclude that
\begin{equation*}
M(r)\ls 16m \frac{M}{R}r+Lr+C_3 LR^{1+\epsilon}
\end{equation*}
or $M(r)\ls L r^{1-\epsilon}$ for any $r<R/2$, where $C_3=C^2_1 2^{m-1}/(m^2)$. Now we choose $r=R^{1+\epsilon}$ to conclude the desired (\ref{equ-4.4}) with $C_2=\max\{16m,C_3+1\}$.   To ensure $r<R/2$, we assume $r<2^{-1/(\epsilon+\epsilon^2)}$. The constant 
$$R_\epsilon:=\min\big\{ \tilde R_\epsilon^{1/(1+\epsilon)},\ (R_0/2)^{1/(1+\epsilon)},\ 2^{-1/(\epsilon+\epsilon^2)}\big\},$$
where $\tilde R_\epsilon$ is given in (\ref{equ-4.7}).
 This finishes the proof.
   \end{proof}

The main result in this subsection is given as follows.
\begin{theorem}\label{thm-4.3}
Let $\Omega\subset X$ be a bounded domain and satisfy a uniformly exterior condition with constant $R_0$. Suppose that $f\in   W^{1,2}(\Omega)$ is a harmonic function on $\Omega$ with boundary data $g\in Lip(\overline{\Omega})$. Suppose $g(z_0)=0$ for some $z_0\in\overline\Omega$. Then for any $\epsilon\in(0,1)$, there exists a number $R'_\epsilon\in (0,\frac 1 2\min\{1,R_0 \})$ (depending only on $\epsilon, N, K$ and ${\rm diam}(\Omega)$) such that for any ball $B_r(x_0)$ with  $x_0\in\partial \Omega$  and $r\in (0,R'_\epsilon)$ it holds 
\begin{equation}\label{equ-4.12}
 \sup_{B_r(x_0)\cap \Omega} |f(x)-f(x_0)|\ls C_{\epsilon}L\cdot r^{1-\epsilon}, 
 \end{equation}
  where the constant $C_{\epsilon}>0$ depends only on $\epsilon, N,K$, ${\rm diam}(\Omega)$, and the constant $L$ is a Lipschitz constant of $g$.  
   \end{theorem}
 
 \begin{proof}
 The maximum principle implies 
$$\max_{\overline\Omega}|f|\ls \max_{\overline\Omega} |g |\ls L\cdot{\rm diam}(\Omega),$$
since $g(z_0)=0$ for some point $z_0$.

For each $j\in\mathbb N$, let $\epsilon_j:=2^{-j}$, and  we define a sequence of numbers $\{a_j\}$ by
\begin{equation}\label{equ-4.13}
a_1=\frac{1}{3},\quad a_{j+1}=\frac{a_j}{1+\epsilon_j}+\frac{\epsilon_j}{1+\epsilon_j} =\frac{a_j+2^{-j}}{1+2^{-j}}.
\end{equation}
It is easy to check 
\begin{equation}\label{equ-4.14}
 a_{j+1}\ls 1-\epsilon_j,\ \quad a_{j+1}\gs a_j \quad {\rm and}\quad \lim_{j\to \infty}a_j=1.
 \end{equation}

Since $-f$ is harmonic function with boundary data $-g$, by replacing $f$ with  $-f$,  and noting $\lim_{j\to\infty}a_j\to 1$, it suffices to show that
  for each $k\in\mathbb N$ there exist $R_k\in(0,\frac{1}{2}\min\{R_0,1\})$ such that 
 \begin{equation}\label{equ-add+2}
 M_{x_0}(r)\ls C_4^kLr^{a_k}
 \end{equation}
 for all $r\in (0,R_k)$ and all $x_0\in\partial\Omega$, where the constant $C_4:= 3C_2({\rm diam}(\Omega)+1)$, and  $C_2$ is as in Lemma \ref{lem-4.2}.
 
 Let $\epsilon_1=1/2$. By using Lemma \ref{lem-4.2} and $\max_{\overline\Omega}|f| \ls L\cdot {\rm diam}(\Omega),$ there exists $R_{\epsilon_1} \in(0,\frac 1 2 \min\{1,R_0\})$ such that  
 \begin{equation*}
M_{x_0}(r)\ls \max\big\{C_2 \cdot(L\cdot{\rm diam}(\Omega)) \cdot r^{1/3} +C_2 L r, \ Lr^{1/2}\big\}\ls  C_4 L \cdot r^{1/3}
\end{equation*}
for all  all $x_0\in\partial\Omega$ and $r\in(0,R_{\epsilon_1})$, $r<1$.

 We now argue by induction. Given $j\in\mathbb N$, we assume that  $M_{x_0}(r) \ls C_4^j L\cdot r^{a_j}$ for all  all $x_0\in\partial\Omega$ and $r\in(0,R_{\epsilon_j})$. By using Lemma \ref{lem-4.2} again,  there exists $R_{\epsilon_{j+1}} \in(0,\frac 1 2\min\{1,R_0\})$ such that  
 \begin{equation*}
 \begin{split}
M_{x_0}(r)&\ls \max\big\{C_2  C_4^jL \cdot(2 r^{\frac{1}{1+\epsilon_j}})^{a_j} r^{\frac{\epsilon_j}{1+\epsilon_j}} + C_2 L r, \ Lr^{1-\epsilon_j}\big\}\\
& =\max\big\{C_2  C_4^jL \cdot 2^{a_j} r^{a_{j+1} } + C_2 L r, \ Lr^{1-\epsilon_j}\big\}\qquad \quad({\rm by}\ \ \eqref{equ-4.13})\\
& \ls (2C_2  C_4^j +C_2)L \cdot   r^{a_{j+1} }, \qquad \quad({\rm by}\ \  \eqref{equ-4.14})
\end{split}
\end{equation*}
for all  $r\in(0,R_{\epsilon_{j+1}})$ and all $x_0\in \partial \Omega$. Notice that $2C_2 C_4^j +C_2\ls C_4^{j+1}$. Now the proof of (\ref{equ-add+2}), and hence the proof of the theorem, is finished. 
\end{proof}

\begin{remark}\label{rem-4.4}
The above estimate implies $u\in C^{1-\epsilon}(\overline{\Omega})$. See the next subsection for the general case of harmonic maps.
\end{remark}

 \subsection{Boundary regularity for harmonic maps}
 
 This subsection will deal with the general case that $Y$ is a $CAT(0)$-space. Let $w\in W^{1,2}(\Omega,Y)$ and let $u\in W^{1,2}_w(\Omega,Y)$ be a  harmonic map with boundary data  $w$.
 
We begin with the following global boundedness.

\begin{lemma}\label{lem-4.5}
Suppose that $w\in C(\overline\Omega,Y)$  and denote by
$${\rm osc}_{\overline{\Omega}}w:=\max_{x,y\in\Omega} d_Y\big(w(x),w(y)\big).$$
Then there exists $P\in Y$ such that $d_Y(P,u(x))\ls{\rm osc}_{\overline{\Omega}}w$  for all  $x\in\Omega$.
\end{lemma} 
\begin{proof}
Notice  that  
 $d_Y(u(x),w(x))\in W^{1,2}_0(\Omega)$  and the triangle inequality imply   $$[d_Y(u(x),P)-d_Y(w(x),P)]^+\in W^{1,2}_0(\Omega)$$ for any $P\in Y$.
Fixing arbitrarily $P \in w(\Omega)$, we have $d_Y(P,w(x))\ls {\rm osc}_{\overline{\Omega}}w,$
and hence 
$$ [d_Y(u(x),P)-{\rm osc}_{\overline{\Omega}}w]^+ \in W^{1,2}_0(\Omega).$$
   Now, by the maximum principle and the sub-harmonicity of $d_Y(u(x),P)$ (see Theorem \ref{thm-3.3}(2)), we get
   $$d_Y(u(x),P)\ls  {\rm osc}_{\overline{\Omega}}w,\quad {\rm a.e.\ in }\ \ \Omega.$$ 
 From Theorem \ref{thm-3.3}(1), we know that $u\in C(\Omega)$. Therefore, we have 
 $d_Y(u(x),P)\ls  {\rm osc}_{\overline{\Omega}}w$ for all $x\in \Omega.$
  \end{proof} 

We now extend the boundary estimate in Theorem \ref{thm-4.3} from harmonic functions to harmonic maps.
 
 \begin{theorem} \label{thm-4.6}
Suppose that $\Omega\subset X$ satisfies a uniformly exterior condition with constant $R_0$ and $w\in Lip(\overline{\Omega},Y)$. Then for any $\epsilon\in(0,1)$ it holds 
\begin{equation}\label{equ-4.15}
 d_Y\big(u(x),w(x_0)\big)\ls C_{\epsilon}   L_w   d^{1-\epsilon}(x,x_0) 
 \end{equation}
 for all  $x_0\in\partial\Omega$ and $x\in\Omega$ with $d(x,x_0)<R'_\epsilon$, where $R'_\epsilon$ and $C_\epsilon$ are as in Theorem \ref{thm-4.3}, and  
  $$L_w:= \sup_{x,y\in\overline\Omega} \frac{d_Y\big(w(x),w(y)\big)}{d(x,y)}.$$ 
  In particular, $u$ is continuous at $x_0$ and $u(x_0)=w(x_0)$.
 \end{theorem}
 \begin{proof}
 Fix any   point $x_0\in \partial\Omega$. We solve the Dirichlet problem 
 $${\bf \Delta} f(x) =0\quad {\rm on}\ \Omega\qquad {\rm and}\qquad f(x)-d_Y(w(x_0),w(x))\in W^{1,2}_0(\Omega).$$
 Notice that, by the triangle inequality, the function $g_{x_0}(x) :=d_Y(w(x_0),w(x))$ is Lipschitz continuous on $\overline\Omega$ with a Lipschitz constant 
 $$L_{g_{x_0}}  \ls  L_w\quad {\rm and}\quad g_{x_0}(x_0)=0.$$
 According to Theorem  \ref{thm-4.3}, we have 
\begin{equation} \label{equ-4.16} 
 \sup_{B_r(x_0)\cap\Omega} |f(x)-f(x_0)|\ls C_{\epsilon}  L_w  r^{1-\epsilon}, 
 \end{equation}
   for any ball $B_r(x_0)$ with  $x_0\in\partial \Omega$  and $r\in (0,R'_\epsilon)$.

By using Theorem \ref{thm-3.3}(2), we know that $d_Y(u(x),w(x_0))-f(x)$ is subharmonic on $\Omega$. Now we consider its boundary value on $\partial \Omega$.
 From 
 $d_Y(u(x),w(x))\in W^{1,2}_0(\Omega)$  and the triangle inequality, we have   $$[d_Y(u(x),w(x_0))-d_Y(w(x),w(x_0))]^+\in W^{1,2}_0(\Omega).$$
 Hence, by $ [d_Y(u(x),w(x_0))-f(x)]^+\ls  [d_Y(u(x),w(x_0))-d_Y(w(x),w(x_0))]^+ + [d_Y(w(x),w(x_0))-f(x)]^+ $, we obtain 
  $$[d_Y(u(x),w(x_0))-f(x)]^+\in W^{1,2}_0(\Omega).$$
   Now the maximum principle implies  that 
   $$d_Y(u(x),w(x_0))\ls f(x),\quad {\rm a.e.\ in }\ \ \Omega.$$ 
  Notice that $u\in C(\Omega)$ (by Theorem \ref{thm-3.3}) and $f\in C(\Omega)$, we get   
   $$d_Y(u(x),w(x_0))\ls f(x),\quad  \forall x\in  \Omega.$$  
By combining this and  (\ref{equ-4.16})   we have
 $$d_Y(u(x),w(x_0))\ls f(x)\ls  C_{\epsilon} L_w r^{1-\epsilon},  \qquad \forall\ x\in B_r(x_0)\cap \Omega$$
   for any ball $B_r(x_0)$ with  $x_0\in\partial \Omega$  and $r\in (0,R'_\epsilon)$.  Hence we have finished the proof.
   \end{proof}

When the boundary data $w$ is H\"older continuous on $\overline{\Omega}$, the same proof of Theorem \ref{thm-4.6} implies the following result.

 \begin{corollary}\label{cor-4.7}
Suppose that $\Omega\subset X$ satisfies an exterior density condition with constant $R_0$, and   $w\in C^\alpha(\overline{\Omega},Y)$ for some $\alpha\in(0,1)$.
  Then,  there exist two  constants $\beta\in(0,\alpha)$ and   $r_0\in (0,R_0)$  such that 
\begin{equation*} 
 d_Y\big(u(x),w(x_0)\big)\ls  C d^{\beta}(x,x_0) 
 \end{equation*}
 for all  $x_0\in\partial\Omega$ and $x\in\Omega$ with $d(x,x_0)<r_0$,  
   where $R_0$ is the number given in the definition of exterior density condition, and $C$ depends only on $N, K, {\rm diam}(\Omega),$ and the modulus of a H\"older continuity of $w$.
\end{corollary}
 
 \begin{proof}
 The same argument is as in the proof of Theorem \ref{thm-4.6}, by replacing Theorem \ref{thm-4.3} with Lemma \ref{lem-2.4}. Notice also that a modulus of the H\"older continuity of $w$ dominates the modulus of the H\"older continuity of $g_{x_0}$ for all $x_0\in\partial\Omega.$
  \end{proof}

Now we are ready to prove the main result.
\begin{proof}[Proof of Theorem \ref{thm-1.5}]
We first prove (B). From Lemma \ref{lem-4.5}, we have $d_Y\big(u(x),P\big)\ls  {\rm osc}_{\overline\Omega}w:=\bar M$ for all $x\in\Omega$.

 For any two point $x,y\in\overline\Omega$ with $d(x,y):=r$.  We can assume 
 $$r<R'_{\epsilon}/20: =R_1,$$
  where $R'_\epsilon$ is as in Theorem \ref{thm-4.6}. Suppose not, from Lemma \ref{lem-4.5}, we get $ d_Y(u(x),u(y))\ls 2 \bar M/R_1\cdot r.$
 
 We split the following argument into three cases: (i) $d(x,\partial \Omega)>4R_1$, (ii) $4r<d(x,\partial \Omega)<5R_1$,  and (iii) $ d(x,\partial \Omega)<5r.$
 
 (i)  $d(x,\partial \Omega)>4R_1$.  It follows   $B_{2R_1}(x)\subset \Omega$, by using Theorem \ref{thm-3.3}(1) to the ball $B_{R_1}(x)$, we have (notice that $y\in B_{R_1}(x)$) 
\begin{equation}\label{equ-4.19}
d_Y\big(u(x),u(y)\big)\ls\frac{C_{N,K,{\rm diam}(\Omega)}}{R_1}\bar M\cdot d(x,y).
\end{equation}
 
  (ii)  $ 4r<d(x,\partial \Omega)<5R_1$. Let $x_0\in\partial \Omega$ such that $d(x,x_0)=d(x,\partial\Omega):=\bar d$. It follows  $B_{\bar d/2}(x)\subset \Omega$, by using Theorem \ref{thm-3.3}(1) to the ball $B_{\bar d/4}(x)$ and Theorem \ref{thm-4.6}, we have
\begin{equation}\label{equ-4.20}
\begin{split}
d_Y\big(u(x),u(y)\big)& \ls\frac{C_{N,K,{\rm diam}(\Omega)}}{\bar d/4} \sup_{B_{\bar d/2}(x)}d_Y((u(x),u(x_0)\big)\cdot d(x,y)\\
&\ls C_{\epsilon,w}  \frac{r^\epsilon }{\bar d^{\epsilon}}\cdot r^{1-\epsilon} \ls C_{\epsilon,w}r^{1-\epsilon}, \qquad ({\rm since}\ \  r\ls \bar d,)
\end{split}
\end{equation}
where $C_{\epsilon,w}$ depends only on $\epsilon,L_w,N,K$ and ${\rm diam}(\Omega)$.

   (iii)  $  d(x,\partial \Omega)<5r$. Let $x_0,y_0$ be two points in $\partial \Omega$ such that $d(x,x_0)=d(x,\partial\Omega)$ and $d(y,y_0)=d(y,\partial\Omega)$.   By using  Theorem \ref{thm-4.6} to both $x_0$ and $y_0$, we have
\begin{equation}\label{equ-4.21}
\begin{split}
d_Y\big(u(x),u(y)\big)&  \ls d_Y\big(u(x),u(x_0)\big)+d_Y\big(u(y),u(y_0)\big)+d_Y\big(u(x_0),u(y_0)\big)\\
&\ls C_{\epsilon,w}\big(d^{1-\epsilon}(x,x_0)+d^{1-\epsilon}(y,y_0)\big)+L_wd(x_0,y_0)\\
&\ls C_{\epsilon,w}(2\cdot (5r)^{1-\epsilon})+L_w\big(d(x,x_0)+d(x,y)+d(y,y_0)\big)\\
&\ls  10C_{\epsilon,w} \cdot r^{1-\epsilon}+11L_w\cdot r,
\end{split}
\end{equation}
where $C_{\epsilon,w}$ depends only on $\epsilon,L_w,N,K$ and ${\rm diam}(\Omega)$.
 
Combining (\ref{equ-4.19})-(\ref{equ-4.21}) in the above three cases, we obtain 
 \begin{equation}\label{equ4.22}
\begin{split}
d_Y\big(u(x),u(y)\big)\ls C\cdot r^{1-\epsilon}, 
\end{split}
\end{equation}
for all $x,y$ with $d(x,y)\ls R_\epsilon'/20$,  where $C_{\epsilon,w}$ depends only on $\epsilon,{\rm osc}_{\overline{\Omega}}w,L_w, R_\epsilon',N,K$ and ${\rm diam}(\Omega)$. This finishes the proof of (B).
  
  The proof of (A) is the same as the above, by replacing Theorem  \ref{thm-4.6} with Corollary \ref{cor-4.7}. The proof is completed.
  \end{proof}

\end{document}